\documentclass{amsart} 
\usepackage{latexsym,amssymb,graphicx,amscd,amsmath}

\newtheorem{thm}{Theorem}[section]
\newtheorem{lem}[thm]{Lemma}
\newtheorem{prop}[thm]{Proposition}
\newtheorem{cor}[thm]{Corollary}

\newcommand{\BZ}{{\mathbb{Z}}}
\newcommand{\BQ}{{\mathbb{Q}}}
\newcommand{\BR}{{\mathbb{R}}}

\newcommand{\la}{{\lambda}}

\newcommand{\Si}{{\Sigma_g}}

\DeclareMathOperator{\id}{id}

\newcommand{\ff}{{\mathfrak{f}}}

\DeclareMathOperator{\Sp}{Sp}

\DeclareMathOperator{\SL}{SL}
\DeclareMathOperator{\sgn}{sgn}
\DeclareMathOperator{\Id}{Id}
\DeclareMathOperator{\Signature}{Sign}
\DeclareMathOperator{\Image}{Image}
\DeclareMathOperator{\kernel}{kernel}

\begin{document}
\title{Two functions on $\Sp(g, \BR)$ }

\author{ Patrick M. Gilmer}
\address{Department of Mathematics\\
Louisiana State University\\
Baton Rouge, LA 70803\\
USA}
\email{gilmer@math.lsu.edu}

\urladdr{www.math.lsu.edu/\textasciitilde gilmer/}

 \thanks{This research was partially supported by NSF-DMS-0905736, NSF-DMS-131191.}

\date{February 20, 2015}

\begin{abstract} We consider two functions on $\Sp(g, \BR)$ with values in the cyclic group of order four $\{\pm 1, \pm i \}$. One was defined by Lion and Vergne 
 \cite{LV}. The other is $-i$ raised to the power given by an integer valued function defined by Masbaum and the author (initially on the mapping class group of a surface) \cite{GM}. 
 We identify these functions when restricted to $\Sp(g, \BZ)$. We conjecture the identity of these functions on $\Sp(g, \BR)$ and prove this conjecture when $g=1$. 
\end{abstract}

\maketitle

\section{\   Introduction} \label{sec.intro}

For $f \in \Sp(g, \BR),$ we describe,  in the next two subsections,  an invariant $s(f)$ of Lion-Vergne which takes values in $\{ \pm1, \pm i \}$ and an  invariant  $n(f)$ of Gilmer-Masbaum which takes values in $\BZ$.  Our main
theorem is:

\begin{thm} \label{main} For $f \in \Sp(g, \BZ),$ 
\begin{equation} \label{eq} s(f) =i^{-n(f)}.\end{equation}
\end{thm}

In the second section we prove this theorem.
We  conjecture this theorem also holds for $f \in \Sp(g, \BR)$. We discuss this conjecture in the third section and prove this conjecture in genus 1. 
In the fourth section, we prove the square of equation \ref{eq} for $f \in \Sp(g, \BR)$.
In the last section, we use the above theorem to give a proof of a known description of the universal central extension of $\Sp(g, \BZ)$ as the inverse image of $\Sp(g, \BZ)$ in the universal  cover of $\Sp(g, \BR)$. 
We also note there that this known description together with other results affords an alternative proof of Theorem \ref{main}.

\subsection {Lion and Vergne's $s(f)$}

Let $V$ be a real vector space equipped with a skew symmetric nonsingular form $\omega$. We refer to  $V$ as a symplectic inner product space.  A lagrangian is a subspace $\lambda$ of $V$ which is equal to its own perpendicular subspace with respect to $\omega.$ 

An oriented vector space is a vector space equipped with an equivalence class of  ordered bases. Two ordered bases are equivalent if the change of basis matrix has positive determinant. To an ordered pair of oriented lagrangians $(\lambda_1,\lambda_2)$ in $V$, Lion and Vergne associated $\epsilon(\lambda_1,\lambda_2)\in\{-1,1\}$, as follows.

Define $h_{\lambda_1,\lambda_2}:\lambda_1\rightarrow\lambda_2^\ast
$ by
 $
h_{\lambda_1,\lambda_2}(x)(y)=\omega(x,y)$. 
The kernel of this map is $\lambda_1\cap\lambda_2$, which we will denote by $\kappa$. 

\subsubsection{  $\epsilon(\lambda_1,\lambda_2)$ in the case $\kappa=0$}
In this case, we have that  $h_{\lambda_1,\lambda_2}$ is invertible. 
Let $\{a_i\}_{i=1,n}$ be an ordered basis for $\lambda_1$, $\{b_i\}_{i=1,n}$ be an ordered basis for $\lambda_2$. We let $\{b_i^\ast\}_{i=1,n}$ be the  ordered basis for  $\lambda_2^\ast$ given by $b_i^\ast(b_j)=\delta_{ij}$. 
One defines $\epsilon(\lambda_1,\lambda_2)$ to be one if and only if  $\{h_{\lambda_1,\lambda_2}(a_i)\}$ and $b_i^\ast$ determine the same orientation on $\lambda_2^\ast$.  
Equivalently $\epsilon(\lambda_1,\lambda_2)= \sgn(\det(\omega(a_i,b_j))).$ Here and below, we let $\sgn(x) =\frac {|x|} {x} \in \{\pm 1\}$   for a non-zero real number $x$. 

\subsubsection{  $\epsilon(\lambda_1,\lambda_2)$ in the case $\kappa\ne 0$}
This case is reduced to the case $\kappa = \{0\}$ as follows.  We can see that $\kappa$ is isotropic, and hence $\kappa^{\perp}\slash\kappa$ acquires an induced    symplectic structure and  $\lambda_1\slash \kappa$, $\lambda_2\slash \kappa$ are lagrangian subspaces of $\kappa^{\perp}\slash\kappa$. 
Choosing an orientation of $\kappa$, we consider the short exact sequences:
\begin{eqnarray}
\label{es}
0\rightarrow \kappa \rightarrow \lambda_i\rightarrow \lambda_i\slash\kappa\rightarrow 0.
\end{eqnarray}
and determine an orientation of $\lambda_i\slash \kappa$, by the rule that an ordered basis for  
$\kappa$ followed by the lift to $\la_i$ of an ordered basis for $\lambda_i\slash \kappa$ is an ordered basis for $\la_i$. 
  Since $\lambda_1\slash\kappa$ and $\lambda_2\slash\kappa$ intersect trivially, 
 $\epsilon (\lambda_1\slash\kappa,\lambda_2\slash\kappa)$ is defined,  and we may define 
$\epsilon(\lambda_1,\lambda_2)=\epsilon (\lambda_1\slash\kappa,\lambda_2\slash\kappa)$.
Here the choice of orientation of $\kappa$ is not important, as this choice appears twice in this construction.

If $\lambda_1,\lambda_2$ are the same lagrangian with the same orientation, then the above prescription asks us to compare two orientations on a zero dimensional vector space. This should be interpreted as follows:
 $\epsilon(\lambda_1,\lambda_2)=1$. Similarly: if $\lambda_1,\lambda_2$ are the same lagrangian but with  opposite orientations, then we 
 take $\epsilon(\lambda_1,\lambda_2)=-1$.
 
 \subsubsection{Definition of s(f) in terms of $\epsilon$ }
 Define $$s( \lambda_1,\lambda_2) = i^{\dim(\lambda_1)  - \dim (\lambda_1 \cap \lambda_2)} \epsilon(\lambda_1,\lambda_2).$$ 
 
 Consider the vector space  $\BR^{2g}$, with the standard basis denoted by $\{p_1, \cdots p_g, q_1, \cdots q_g \}$ and equipped with the standard symplectic form given  by  $\omega(p_i,p_j)=\omega(q_i,q_j)= 0$ and  $\omega(p_i,q_j)= -\omega(q_i,p_j) =\delta_{ij}$.  
 The Lie group of  isometries of this symplectic inner product space is called the symplectic group and is denoted  $\Sp(g, \BR)$. 
 Let $\lambda_0$ be the lagrangian spanned by $\{p_i\}$.
 If $f \in \Sp(g, \BR)$,   define  $$s(f)=s(\lambda_0, f(\lambda_0)).$$
 Here we  give $\lambda_0$ an arbitrary orientation.  Since this orientation enters the computation twice, it does not effect the result.

 \subsection{Gilmer-Masbaum's $n(f)$}

Let  $f:V \rightarrow V$ be an isometry.
 Turaev \cite{T2})\cite[2.1,2.2]{T3}) defined a non-singular  bilinear form $ \star_{f} $  on $(f-1)V$
by \[a \star_{f} b =\omega((f-1)^{-1}(a), b ). \] 
Here, $\omega((f-1)^{-1}(a), b )$ means $\omega(x, b)$ where $x$ is any
element of $(f-1)^{-1}(a) $.

The  determinant of a matrix for $\star_f$ with respect to a basis of $(f-1)V$ will be denoted $\det(\star_{f})$. Thus $\sgn[\det(\star_{f})]$  
 will take values in $\{\pm 1\}$.  If $f= \Id$, $(f-1)V=0$, and we let $\sgn[\det(\star_{\Id})]=1$.
 
 According to \cite[Lemma 6.4]{GM},  if  $\la\subset V$ is a lagrangian, then the
  restriction of the form $\star_{f}$ to $\lambda \cap (f-1)V$ is symmetric. This form is denoted $\star_{f, \la}$. 
  Thus $\star_{f,\la}$ has a signature.
  
In the above situation, one defines
\begin{equation} \label{me} n_\la(f) =
\Signature(\star_{f,\lambda})
-\dim((f-1)V)-
\sgn[\det(\star_{f})]  +1~.
\end{equation}

For $f \in \Sp(g, \BR)$, 
let   $$n(f)=n_{\lambda_0}(f).$$

We note that $n_\la(f)$ was defined in \cite{GM} for $f$ in the mapping class group of a surface (using $H_1(\Si, \BQ)$ with a chosen lagrangian $\la$ of this rational vector space). The terms in the formula for $n_{\la_0}(f)$ make  perfect sense for  $f \in \Sp(g, \BR)$, so we can make this definition.

We also consider the free $\BZ$ module generated by  $\{p_1, \cdots p_g, q_1, \cdots q_g \}$ which we identify with $\BZ^{2g} \subset \BR^{2g}$.
 The form $\omega$ restricts to a unimodular $\BZ$-valued form. By $\Sp(g, \BZ)$, we the mean the group of isometries of this symplectic inner product space over $\BZ$.  We have that $\Sp(g, \BZ) \subset \Sp(g, \BR)$. So we may also restrict $s$ and  $n$ to  $\Sp(g, \BZ)$.

We would like to thank Gregor Masbaum for useful conversations.

\section{Comparing characters on a central extension of $\Sp(g, \BZ)$}

Given three lagrangians $\la_1$, $\la_2$, $\la_3$ of $(V,\omega)$, there  is a Maslov index $\mu(\la_1,\la_2,\la_3) \in \BZ$.
This can be defined as the signature of the symmetric bilinear form on $(\la_1+\la_2)\cap \la_3$ defined by $B(a,b)=\omega(x,b)$ where
$a,b \in (\la_1+\la_2)\cap \la_3$, $x \in \la_2$, and $a-x \in \la_1.$ As noted in \cite{T2,T3}, this is equivalent to the definition given by Kashiwara and used in \cite{LV}. As we use both  \cite{T2,T3} and \cite{LV},   some our results depend  on this identification which can be seen  for instance using \cite[Thm 8.1]{CLM}.

Lion and Vergne \cite[1.6.14]{LV} use the Maslov index to specify a certain central extension $\widetilde {\Sp(g, \BR)}$ 
by  $\BZ$ of  $\Sp(g, \BR)$. 
One defines $$\widetilde {\Sp(g, \BR)}= \{ (f,m) | f \in \Sp(g, \BR), m \in \BZ \}$$  with multiplication;
\begin{equation} \label{cocy}(f_1,m_1) \cdot (f_2, m_2)= (f_1 f_2, m_1 +m_2 + \mu(\la_0,f_1(\la_0),f_1f_2 (\la_0)) )\end{equation}
Thus $\widetilde {\Sp(g, \BR)}$ is the central extension of ${\Sp(g, \BR)}$ specified by the 2-cocycle $\nu$ where
$\nu(f_1,f_2)=  \mu(\la_0,f_1(\la_0),f_1f_2 (\la_0)).$ 
According to \cite[1.7.11]{LV}, the formula $\mathfrak s(f,m)= i^m s(f)$ defines a character on the group $\widetilde{\Sp(g, \BR)}$.

One can define an extension $\widetilde{\Sp(g, \BZ)}$ of  $\Sp(g, \BZ)$ by the same procedure as used for  $\Sp(g, \BR)$, and one obtains the pull back  by the inclusion $\iota  : {\Sp(g, \BZ)} \rightarrow {\Sp(g, \BR)}$ of the extension $\widetilde{\Sp(g, \BR)}$ over $\Sp(g, \BR)$. We have the following commutative diagram with exact
rows. 
\[
 \begin{CD}
0 @>>> \BZ @>>>  \widetilde{\Sp(g, \BZ)}  @>>>  {\Sp(g, \BZ)}@>>> 1    \\
@.          @V=VV         @VV{\tilde \iota}V                                     @VV \iota V        @.   \\
0 @>>> \BZ @>>> \widetilde{\Sp(g, \BR)}  @>>>  {\Sp(g, \BR)}  @>>> 1 .   
\end{CD}
\] 

We define $\mathfrak r: \widetilde{\Sp(g, \BR)} \rightarrow  \{ \pm 1, \pm i \}$ by $\mathfrak r(f,m)= i^{n(f)-m}$.  
  We need the following lemma whose proof we delay.

\begin{lem}\label{char} The function   $\mathfrak r \circ  \tilde \iota$   is a character on $ \widetilde{\Sp(g, \BZ)}$ with values in $\{ \pm 1, \pm i \}$.
\end{lem}

\begin
{proof} [Proof of Theorem \ref{main} modulo  Lemma \ref{char}] Given that  $\mathfrak r \circ  \tilde \iota$ and $\mathfrak s  \circ  \tilde \iota$ are characters, it follows that $\mathfrak t(m,f)= \mathfrak r(f,m) \mathfrak s(f,m)$ defines  a character on $\widetilde{\Sp(g, \BZ)}$. This character vanishes on the central element $(1,\id) \in \widetilde{\Sp(g, \BZ)}$. Thus
$\mathfrak t$  induces a well defined  character on ${\Sp(g, \BZ)}$. 
According to \cite[Thm 5.1]{P}, if  $g \ge 3$, ${\Sp(g, \BZ)}$ is perfect. So  the induced character is trivial. It follows that $\mathfrak t$  is identically one. 
Thus    $s(f) =i^{-n(f)}$ for $f \in  {\Sp(g, \BZ)}$ if  $g \ge 3$. But both $s$ and  $n$ remain unchanged upon the stabilization  ${\Sp(g, \BZ)} \rightarrow {\Sp(g+1, \BZ)}$ given by direct summing a $2 \times 2$ identity matrix. Thus $s(f) =i^{-n(f)}$ for low genus as well.   
\end{proof}

We will think of $s$  and $n$ as $1$-cochains on ${\Sp(g, \BR)}$.   
We write $n$ as  $-j-k$ where  $j$ and $k$ are the two $1$-cochains (the notation is chosen to be consistent with \cite{GM}).
\begin{equation}
j(f)= -\Signature(\star_{f,\lambda_0}) \quad \text{ and } \quad  k(f)= \dim\left( \Image(f-\Id) \right) + 
 \sgn[\det(\star_f)] 
-1  \label{jk}\end{equation}

 \begin{prop}[Turaev \cite{T2,T3}] Let  $f_1,f_2 \in {\Sp(g, \BR)}$,
\begin{equation}  x \star_{f_1,f_2} y =\omega \left(  (    f_1-1)^{-1}x+ (f_2-1)^{-1} x +x, y \right ) \notag
\end{equation} defines a symmetric bilinear form on $\Image(f_1-1)\cap \Image (f_2-1).$ 
\end{prop}

Consider the 2-cochain given by 
\begin{equation}\label{phi} \phi(f_1,f_2)=\Signature  ( \star_{f_1,f_2} ) .\end{equation}  Recall the coboundary of a 1-cochain $c$ is given by
$(\delta c) (g,h)= c(g)+c(h)-c(gh)$. We need:

 \begin{thm}[Turaev \cite{T2,T3}] \begin{equation} \label{tur}  \delta k \equiv \varphi \pmod{4} .\end{equation}
\end{thm}

Our next result uses some topology. Let $\Gamma_g$ denote the mapping class group of a closed surface $\Si$ of genus $g$.
We may pick an identification of $H_1(\Si)$ with $\BZ^{2g}$ so that the intersection pairing on $H_1(\Si)$ agrees with $\omega$. Then we have a surjection
$h: \Gamma_g \rightarrow {\Sp(g, \BZ)}$ which sends a mapping class $\ff$ to the map it induces on homology.  We also identify $H_1(\Si, \BR)$ with  $\BR^{2g}.$ We pick a handlebody $\mathcal H_g$ with boundary $\Si$ such that $\la_0$ under this identification is the kernel of the map
$H_1(\Si, \BR) \rightarrow H_1({\mathcal H_g}, \BR)$. Proposition \ref{sigid} is essentially Walker's theorem \cite[p. 124]{W} \cite[Thm 8.10]{GM}  together with  \cite[Prop 8.9]{GM} which identifies  the signatures of certain manifolds  appearing in following proof with $\Signature (\star_{f, \la_0})$ for various $f$.

\begin{prop}\label{sigid} Let  $f_1,f_2 \in {\Sp(g, \BZ)}$, 
\begin{equation}\label{wtj}\Signature (\star_{f_1,f_2} ) -\mu(\la_0, f_1(\la_0), f_1 f_2 (\la_0)) + \Signature (\star_{f_1\circ f_2, \la_0}) 
 -\Signature (\star_{f_1, \la_0})  -\Signature (\star_{f_2, \la_0})=0 
\end{equation} \end{prop}

\begin{proof} 
Given  $f_1,f_2 \in {\Sp(g, \BZ)}$, we pick $\ff_1,\ff_2 \in \Gamma$, with $h(\ff_i)= f_i$. Then we use $\ff_1$,  $\ff_2$ and $\mathcal H_g$ to construct five 4-manifolds with boundary as in \cite[proof of Thm 8.10]{GM}.  Using identities appearing in \cite{GM}, the signatures of each of these manifolds are identified with the terms that appear on the left hand side of equation \ref{wtj}. Then we glue together the five 4-manifolds along whole components of their boundaries to obtain a closed 4-manifold. By Novikov additivity, this closed 4-manifold has signature given by the left hand side of equation \ref{wtj}.  This closed 4-manifold is then shown to be the boundary of a five manifold, as in \cite[proof of Thm 8.10]{GM}, and thus have vanishing signature. \end{proof}

Because these constructions require that $f_1$ and $f_2$ be the maps on the homology of a surface induced by surface automorphisms, the above proof does not extend to the case  
$f_1, f_2 \in \Sp(g, \BR)$.

\begin
{proof} [Proof of Lemma \ref{char}] The claim is easily seen to be equivalent to the following identity involving 2-cocycles of ${\Sp(g, \BZ)}$:
\begin{equation} \label{eq3} \mu(\la_0,f_1(\la_0),{f_1}{f_2} (\la_0) ) + \delta n_{\la_0}({f_1},{f_2})=0 \pmod{4}. \notag \end{equation}
Using equations \ref{jk}, \ref{phi},  \ref{tur} and \ref{wtj}, and letting   $\equiv$ denote equality modulo 4, 
\begin{align}\delta(n_{\la_0})(f_1,f_2)&=   -\delta(j)(f_1,f_2)-\delta(k)(f_1,f_2) \\ 
  & \equiv   \Signature (\star_{f_1, \la_0}) + \Signature (\star_{f_2, \la_0}) -\Signature (\star_{f_1\circ f_2, \la_0}) 
-\Signature (\star_{f_1,f_2} )  \\
 &= -\mu(\la_0, f_1(\la_0), f_1 f_2 (\la_0)).
 \end{align}
\end{proof}

We remark that Lemma \ref{char} and its proof are closely related to \cite[Thm 6.6]{GM} and its proof. 
 
 \section{On the conjecture that Theorem 1.1 holds for $f \in  \Sp(g, \BR)$}

By the argument for Lemma \ref{char}, we have:
\begin{lem} If equation \ref{wtj} holds modulo  four for all $f_1,f_2 \in {\Sp(g, \BR)}$, then  
$\mathfrak r$ is a character with values in $\{ \pm 1, \pm i \}$ on the group $\widetilde{{\Sp(g, \BR)}}$  
 \end{lem}

\begin{prop} If equation \ref{wtj} holds modulo four for all $f_1,f_2 \in {\Sp(g, \BR)}$, then  
equation \ref{eq} holds for all $f$ in ${{\Sp(g, \BR)}}$  
 \end{prop}

 \begin{proof} We use essentially the same argument as in the proof of Theorem \ref{main} except we do not need to stabilize as  ${\Sp(g, \BR)}$ is perfect even for low $g$. \end{proof}
 
 \begin{prop} Equation \ref{eq} holds for all $f$ in ${{\Sp(1, \BR)}}$.\end{prop}
 
  \begin{proof}(sketch) We note ${{\Sp(1, \BR)}}= {{\SL(2, \BR)}}$. One easily has \cite[1.8.4]{LV} that, if $a \ne 0$, then
  $s\left ( \bmatrix  a&b \\ 0&a^{-1}\\  \endbmatrix \right) = \sgn(a) $, and if $c \ne 0$, then $s\left( \bmatrix  a&b \\ c&d\\  \endbmatrix \right)= \sgn(c) i.$
 To complete the proof  one only needs to calculate $n$ for the following cases. The results of these calculations (especially modulo 4) can be grouped together more efficiently, but the calculations proceed  differently in each case listed.
  \begin{itemize}
  \item For $a \ne 0$ and $a \ne 1$, $ n\left(\bmatrix  a&b \\ 0&a^{-1}\\  \endbmatrix \right) = \begin{cases}0 &\mbox{if } a>0  \\
-2 & \mbox{if } a<0 . \end{cases} $
 \item For $b \ne 0$, $ n\left( \bmatrix  1&b \\ 0&1\\  \endbmatrix \right) = 0$
\item  $n\left( \bmatrix  1&0 \\ 0&1\\  \endbmatrix \right) = 0$
\item For $c \ne 0$ and $b = (a-1)(d-1)c^{-1}$, 

$ n\left( \bmatrix  a&b \\ c&d\\  \endbmatrix \right) = \begin{cases}-1 &\mbox{if } c>0  \\
1 & \mbox{if } c<0 . \end{cases} $
\item For $c \ne 0$ and $b \ne  (a-1)(d-1)c^{-1}$, $$ n\left( \bmatrix  a&b \\ c&d\\  \endbmatrix  \right)= \begin{cases}-1 &\mbox{if } c>0 \mbox{ and  } 
(a-1)(d-1)>bc  \\
-3 &\mbox{if } c<0 \mbox{ and } 
(a-1)(d-1)>bc  \\
-1 &\mbox{if } c>0 \mbox{ and } 
(a-1)(d-1)<bc  \\
1 &\mbox{if } c<0 \mbox{ and } 
(a-1)(d-1)<bc.
\end{cases} $$
  \end{itemize}

   \end{proof}

 \section{The square of  equation  \ref{eq} \label{s2}}
 
 We can obtain that the square of equation \ref{eq} is valid for $f \in {{\Sp(g, \BR)}}$. Moreover the proof is much simpler 
 than the proof of Theorem \ref{main}.
 
 \begin{prop}If $f \in {{\Sp(g, \BR)}}$,  ${(s(f))}^2=(-1)^{n(f)}.$ 
\end{prop}
\begin{proof}
 From the definition of $s$, one easily has that $${(s(f))}^2= (-1)^{g+ \dim(\la_0 \cap f(\la_0))}.$$
 From the definition of $n_{\la_0}$, one easily has that  $$(-1)^{n(f)}= (-1)^{\Signature(\star_{f,\lambda_0})-\dim(\Image(f-1))}.$$
  By \cite[Propostion 7.3]{GM} (whose proof is valid for $f \in {{\Sp(g, \BR)}}$), 
  $${g+ \dim(\la_0 \cap f(\la_0))} = \Signature(\star_{f,\lambda_0})-\dim(\Image(f-1)) \pmod{2}.$$
\end{proof}

 \section{Universal covers and universal central extensions \label{fur}}
 
 Theorem \ref{main} and certain results in \cite{LV,GM} suggest a description 
  of the universal central extension of ${\Sp(g, \BZ)}$ as a subgroup of the universal covering group of ${\Sp(g, \BR)}$ which we now give as Corollary \ref{uce}.  We were not able to find a  proof in the literature, but we are informed that this result is folklore.

 As $\pi_1({\Sp(g, \BR)}) = \BZ$, the universal covering group of ${\Sp(g, \BR)}$ is an infinite cyclic cover which we will denote by 
 $\mathcal{UC}({\Sp(g, \BR)})$.   
 Consider the central extensions $E_g$ of  ${\Sp(g, \BZ)}$ defined in the following diagram:
 \[
 \begin{CD}
0 @>>> \BZ @>>>  E_g= \mathcal{UC}({\Sp(g, \BR)})\cap \pi^{-1}   {\Sp(g, \BZ)}  @>>>  {\Sp(g, \BZ)}@>>> 1    \\
@.          @V=VV         @VVV                                     @VV \iota V        @.   \\
0 @>>> \BZ @>>> \mathcal{UC}({\Sp(g, \BR)})  @>\pi >>  {\Sp(g, \BR)}  @>>> 1 .   
\end{CD}
\] 
For $g=1$, this  extension is  described by  Milnor in  \cite[Thm 10.5]{M} and a remark following it. 

  \begin{cor}\label{uce}  If $g  \ge 4$, $E_g$ is a universal central extension of ${\Sp(g, \BZ)}$; i.e. a universal central extension of ${\Sp(g, \BZ)}$ is
given by  the inverse image of ${\Sp(g, \BZ)}$ under the universal covering projection $\mathcal{UC}({\Sp(g, \BR)}) \xrightarrow{\pi} {\Sp(g, \BR)}$.
  \end{cor} 
  
  \begin{proof}
  One may define  $\widetilde {\Gamma_g}= \{ (\ff,m) | \ff \in \Gamma_g, m \in \BZ \}$ with multiplication given by 
\begin{equation} \label{cocy2}(\ff_1,m_1) \cdot (\ff_2, m_2)= (\ff_1 \ff_2, m_1 +m_2 + \mu(\la_0,{\ff_1}_*(\la_0),{\ff_1}_*{\ff_2 }_*(\la_0)) ).\end{equation}
Compare \cite{W,GM}.   
 One has the following commutative diagram with exact
rows. 
\[
 \begin{CD}
 0 @>>> \BZ @>>> \widetilde\Gamma_g @>>> \Gamma_g @>>> 1    \\
 @.          @V=VV         @VV{\tilde h}V                                     @VVhV           @.   \\
0 @>>> \BZ @>>>  \widetilde{\Sp(g, \BZ)}  @>>>  {\Sp(g, \BZ)}@>>> 1    \\
@.          @V=VV         @VV{\tilde \iota}V                                     @VV \iota V        @.   \\
0 @>>> \BZ @>>> \widetilde{\Sp(g, \BR)}  @>>>  {\Sp(g, \BR)}  @>>> 1 .   
\end{CD}
\] 

It was noted in \cite[Prop 10.1]{GM}, that if $g \ge 4$,   then  $H_2({\Gamma_g})=\BZ$ and the subgroup of $\widetilde \Gamma_g$ given by
$ \{ (\ff,m) | n(\ff )=m \pmod{4} \}$  is a universal central extension of $\Gamma_g$.
Thus the cohomology class which classifies this extension  generates $H^2(\Gamma_g)$.   
We note that $ \{ (\ff,m) | n(\ff )=m  \pmod{4}\}= \kernel(\mathfrak r \circ \tilde i \circ \tilde h)$.  
According to \cite[Thm 5.1]{ P} (still assuming $g \ge 4$), ${\Sp(g, \BZ)}$ is perfect  and $H_2({\Sp(g, \BZ)})=\BZ$. So ${\Sp(g, \BZ)}$ has a universal central extension by $\BZ$.
By Putman \cite[Lemma 7.5]{P}, $h$ induces an isomorphism $H_2( \Gamma_g) \rightarrow H_2({\Sp(g, \BZ)})$. Thus the pull-back of a universal central extension of $\Gamma_g$ to an extension of ${\Sp(g, \BZ)}$ is a universal central extension of ${\Sp(g, \BZ)}$, which we will denote by $\mathcal{UCE}({\Sp(g, \BZ)})$.
Thus the index 4 subgroup of $\widetilde{\Sp(g, \BZ)}$ given by $\kernel(\mathfrak r \circ \tilde \iota)$  
is $\mathcal{UCE}({\Sp(g, \BZ)})$. 

According to \cite[p. 94]{LV}\footnote{This result is not  explicitly proved in  \cite{LV}  but it follows from
\cite[1.9.16]{LV}}, the kernel of $\mathfrak s$ in $\widetilde{\Sp(g, \BR)}$ is  $\mathcal{UC}({\Sp(g, \BR)})$.
Thus the $\kernel(\mathfrak s \circ \tilde \iota)$ is $\mathcal{UC}({\Sp(g, \BR)})\cap \pi^{-1}   {\Sp(g, \BZ)}.$

By the proof of Theorem \ref{main}, $\mathfrak r \circ \tilde \iota= \mathfrak s \circ \tilde \iota.$
Thus $\mathcal{UCE}({\Sp(g, \BZ)})  = \mathcal{UC}({\Sp(g, \BR)})\cap \pi^{-1}   {\Sp(g, \BZ)}.$ 
\end{proof}

If one takes the known  Corollary \ref{uce} as a starting point, one can 
reverse the above argument and view Theorem \ref{main} as a corollary of it's corollary. We note that
this proof  makes implicit use of Lemma \ref{char}.

\begin{proof}[ Proof that Corollary \ref{uce} 
implies
Theorem \ref{main}] \quad
\newline
\noindent Using the stabilization argument used in the proof of Theorem \ref{main}, it is enough to show 
$\mathfrak r  \circ \tilde \iota$ and ${\mathfrak s  \circ \tilde \iota}$ are reciprocals  for large $g$. Starting with the Gilmer-Masbaum description of the universal central extension of the mapping class group in terms of $n$, one concludes  as in the above proof that 
$ \kernel(\mathfrak r \circ \tilde \iota)$ is a universal central extension of ${\Sp(g, \BZ)}$. Starting from the Lion-Vergne
description of the universal cover of $\Sp(g, \BR)$ in terms of $s$, one obtains 
$ \kernel(\mathfrak s \circ \tilde \iota)=\mathcal{UC}({\Sp(g, \BR)}) \cap \pi^{-1}   {\Sp(g, \BZ)}$ which by hypothesis is a universal central extension of  ${\Sp(g, \BZ)}$. Thus 
  $\kernel(\mathfrak r \circ \tilde \iota)=\kernel(\mathfrak s\circ \tilde \iota)$. It follows that either  
  $\mathfrak r \circ \tilde \iota$ and $\mathfrak s \circ \tilde \iota$ agree or one is the  reciprocal of the other. Consideration of
  the $i^{-m}$ and $i^{m}$ terms in the defining formulas for  $\mathfrak r$ and $\mathfrak s$ shows that they must be reciprocals.
\end{proof}

The  following corollary  of Theorem \ref{main} generalizes a remark made by Milnor in the case $g=1.$

\begin{cor} The fundamental group of the orbit space ${\Sp(g, \BR)}/{\Sp(g, \BZ)}$ is
$E_g$. In particular, if $g  \ge 4$, the fundamental group of  ${\Sp(g, \BR)}/{\Sp(g, \BZ)}$ is isomorphic to 
a universal central extension of ${\Sp(g, \BZ)}$.
\end{cor} 

\begin{proof} We extend the diagram before the statement of Corollary \ref{uce} by adding the orbit spaces of two actions by subgroups. 
 \[
 \begin{CD}
0 @>>> \BZ @>>>  E_g  @>>>                             {\Sp(g, \BZ)}@>>> 1    \\
@.          @V=VV         @VVV                                     @VV \iota V        @.   \\
0 @>>> \BZ @>>> \mathcal{UC}({\Sp(g, \BR)})  @>>>  {\Sp(g, \BR)}  @>>> 1 \\
@.       @.                  @VVV                                                   @VVV    \\
\     @.       \     @. \mathcal{UC}({\Sp(g, \BR)}) / E_g    @>\beta >>       {\Sp(g, \BR)}/ {\Sp(g, \BZ)}    @.    \      
\end{CD}
\] 
As $\mathcal{UC}({\Sp(g, \BR)})$ is simply connected,  $\pi_1(\mathcal{UC}({\Sp(g, \BR)})/E_g) \approx E_g.$
Since the induced map $\beta$ is a homeomorphism,
 $\pi_1({\Sp(g, \BR)}/ {\Sp(g, \BZ)})  \approx E_g$.
\end{proof}
 
 \section{Final Comments}
 
Central extensions of the mapping class group are used to upgrade projective representations arising in topological quantum field theory (TQFT) by honest representations \cite{W, MR}. More  generally an extension of the three dimensional cobordism category is used to remove the projective ambiguity of TQFT maps induced by more general cobordisms than mapping cylinders \cite{W,T}. An index two subcategory of the extended cobordism category \cite{G} proved useful in demonstrating that certain projective modules associated to surfaces by an integral version TQFT are free. In  \cite{GM}, the function $n$ was defined in order to describe an index four subgroup of the extended mapping class group. This allowed Masbaum and the author to define modular representations of the unextended mapping class group.
 In \cite[ Remark 7.5]{GM},  it is asked  whether there is a corresponding  index  four subcategory of the 3-dimensional extended cobordism category. As $s$ gives a very different description of this same index four subgroup, it is plausible to hope that Theorem \ref{main} might help answer this question. As a tentative step in this direction, Wang \cite{Wa} makes use of Theorem \ref{main} to define a version of $n$  for connected extended cobordisms  which have been further enhanced with a choice of orientation for the lagrangians  that are part of  the extended structure. This version of $n$ agrees with $n$ when applied to mapping cylinders.


\begin{thebibliography}{CLM}


\bibitem[CLM]{CLM}
{S. Cappell, R. Lee, E. Miller}.
{On the Maslov Index},
{Com. Pure and Appl. Math.},
{\bf 47},
{121--186},
{(1994)}.





\bibitem[G]{G} P. Gilmer. Integrality for TQFTs. {\ em Duke Math. J.} {\bf 125}, 389Ð413. (2004)

\bibitem[GM]{GM} 
{P. Gilmer,  G. Masbaum}.
{Maslov index, Lagrangians, Mapping Class Groups and TQFT}, 
 Forum Math., {\bf 25}, no. 5, 1067--1106,  (2013)


\bibitem[LV]{LV}
{G.Lion, M. Vergne}.
{\em The Weil Representation, Maslov Index and Theta Series},
{Progress in Mathematics 6},
{Birkhauser},{(1980)}.

\bibitem[MR]{MR} G.Masbaum, J.Roberts.   On central extensions of mapping class groups. {\em Math. Ann.}{\bf 302}, 131Ð150,(1995)

\bibitem[M]{M} J. Milnor. { \em Introduction to algebraic K-theory}, Annals of Mathematics Studies, No. 72. Princeton University Press, Princeton, N.J. (1971). 

\bibitem[P]{P} {A. Putman}.  The Picard group of the moduli space of curves with level structures.{ \em Duke Math. J.} {\bf 161}, 623-674, (2012).




\bibitem[T1]{T2} {\sc V. Turaev.}
A cocycle of the symplectic first Chern class and Maslov indices.
{\em Funktsional. Anal. i Prilozhen.} {\bf 18}, no. 1, 43--48,  (1984). 

\bibitem[T2]{T3}{\sc V. Turaev.}
The first symplectic Chern class and Maslov indices,
{\em Journal of Soviet Mathematics} {\bf 37}, 1115-1127, (1987).

\bibitem[T3]{T}{\sc V. Turaev.}
{ \em Quantum invariants of knots and 3-manifolds}. Second revised edition. de Gruyter Studies in Mathematics, 18. Walter de Gruyter \& Co., Berlin, 2010.


\bibitem[W]{W} {\sc K. ~Walker.} {On Witten's $3$-manifold invariants.}  Preliminary version, \\ http://canyon23.net/math/, (1991) 


\bibitem[Wa]{Wa} {\sc X. ~Wang.} {Extra Structures on Three-Dimensional Cobordisms.}, LSU thesis 2013, http://etd.lsu.edu/docs/available/etd-07032013-142703//

\end{thebibliography}
\end{document}